\documentclass[12pt,a4paper]{amsart}

\usepackage{graphics,epic}
\usepackage{amsmath,amssymb, amsthm,mathrsfs}
\usepackage[all,2cell]{xy}

\newtheorem{theorem}{Theorem}[section]
\newtheorem*{theorem*}{Theorem}

\newtheorem{lemma}[theorem]{Lemma}
\newtheorem{proposition}[theorem]{Proposition}

\newtheorem*{conjecture*}{Conjecture}

\newtheorem{example}[theorem]{Example}
\newtheorem{remark}[theorem]{Remark}

\newcommand{\ol}[1]{\overline{#1}}

\newcommand{\opname}[1]{\operatorname{\mathsf{#1}}}

\renewcommand{\mod}{\opname{mod}\nolimits}

\newcommand{\dimv}{\underline{\dim}\,}

\newcommand{\ind}{\opname{ind}}

\newcommand{\Z}{\mathbb{Z}}
\newcommand{\N}{\mathbb{N}}
\newcommand{\Q}{\mathbb{Q}}
\newcommand{\C}{\mathbb{C}}

\renewcommand{\P}{\mathbb{P}}

\newcommand{\Hom}{\opname{Hom}}
\newcommand{\go}{\opname{G_0}}

\newcommand{\Ext}{\opname{Ext}}

\newcommand{\Aut}{\opname{Aut}}

\newcommand{\ca}{{\mathcal A}}

\newcommand{\ch}{{\mathcal H}}

\newcommand{\ct}{{\mathcal T}}

\newcommand{\E}{{\mathbb E}}

\renewcommand{\hat}[1]{\widehat{#1}}

\begin{document}

\title{Quantum dilogarithm identities and cyclic quivers}

\date{Last modified on \today}
\author{Changjian Fu}
\address{Changjian Fu\\Department of Mathematics\\SiChuan University\\610064 Chengdu\\P.R.China}
\email{changjianfu@scu.edu.cn}

\author{Liangang Peng}
\address{
Liangang Peng\\ Department of Mathematics\\SiChuan
University\\610064 Chengdu\\P.R.China} \email{penglg@scu.edu.cn}
\maketitle
\begin{abstract}
We study quantum dilogarithm identities for cyclic quivers following Reineke's idea via Ringel-Hall algebra approach.  For any given discrete stability function for the cyclic quiver $\Delta_n$ with $n$ vertices, we obtain certain cyclic quantum dilogarithm identities of order $n$ in the sense of Bytsko and Volkov.
\end{abstract}
\section{Introduction}

Quantum dilogarithm which is a $q$-deformation of the classical dilogarithm function, appears naturally in many different branches of mathematics such as discrete dynamic system~\cite{Schutzenberger53,Faddeev-Kashaev94,Faddeev-Volkov}, cluster theory~\cite{Fock-Goncharov09,Keller10} and motivic Donaldson-Thomas invariants~\cite{Kontsevich-Soibelman08}, etc. For more details and the history of quantum dilogarithm, we refer to ~\cite{Fock-Goncharov09,Kashaev-Nakanishi11} and the references given there.

Inspired
by Kontsevich-Soibelman's frame work~\cite{Kontsevich-Soibelman08} of motivic Donaldson-Thomas invariants for $3$-Calabi-Yau categories, Reineke~\cite{Reineke2010}
studied some factorization formula for certain automorphism of
Poisson algebra assocaited with a finite  acyclic quiver by using Ringel-Hall algebra approach, which is related to
the wall-crossing formula in ~\cite{Kontsevich-Soibelman08}.
 Keller~\cite{Keller10} reformulated
Reineke's Theorem~\cite{Reineke2010} for Dynkin quiver  with discrete stability structure in the setting
of quantum dilogarithm, which can be interpreted as stating that the
refine DT-invariant as well as non commutative DT-invariant of a
Dynkin quiver are well-defined. It has been conjectured in ~\cite{Keller10} that the statement for Dynkin quivers may be generalized to any Jacobian algebras arising from quivers with polynomial potentials.

One of the key ingredients in Reineke's work~\cite{Reineke2010,Keller10} is the existence of the so-called integration map from the opposite completed Ringel-Hall algebra to a suitable quantum torus. Such an integration map always exists for any hereditary abelian categories. Hence Reineke's work may be generalized to quiver with oriented cycles.
Then this note is devoted to study the quantum dilogarithm identities for cyclic quivers.
 In contrast to the case of acyclic quivers, there does not exist a discrete stability function in the sense of Keller~\cite{Keller10}(which is called {\it completely discrete} in our setting, {\it cf.} Section~\ref{ss:defintion-stability}) for a cyclic quiver, which is very useful to derive quantum dilogarithm identities via Ringel-Hall algebra approach. Fortunately, our result shows that the discrete stability function in the sense of physics  is enough to apply the Ringel-Hall algebra approach for cyclic quivers.

The paper is organized as follows: in Section~\ref{ss:defintion-stability}, after given the definition of  discrete stability function and completely discrete stability function, we give a characterization of stable objects for a given discrete stability function for cyclic quivers(Theorem~\ref{t:discrete-structure}), which enables us to obtain quantum dilogarithm identities for cyclic quivers in Section~\ref{ss:quantum-dilogarithm identities}(Theorem~\ref{t:quantum-identities}).
In~\cite{BytskoVolkov2013}, Bytsko and Volkov introduced cyclic quantum dilogarithm  and obtained certain cyclic quantum dilogarithm identities of order $3$. In Section~\ref{ss:cyclic quantum}, we prove that for any given discrete stability function, one obtains certain cyclic quantum dilogarithm identities of order $n$(Theorem~\ref{t:cyclic quantum-identities}). Section~\ref{s:examples} is devoted to study certain examples of quantum dilogarithm identities related to Jacocian algebra of type $A_2$ and quiver algebras of type $A_{n-1}$.

{\bf Acknowledgments.} This work has been presented by the first-named author on the seminar of representation theory of algebra in Nanjing Normal University, April 19-20. He would like to thank Qunhua Liu and Jiaqun Wei for their invitation and hospitality. This work was partially supported by grants of NSF of China(No. 11001185, 2011CB808003).

\section{Discrete stability functions for cyclic quivers}~\label{ss:defintion-stability}
\subsection{Stability functions on abelian categories}
Let $\ca$ be an abelian category over a field $k$. A {\it stability function} on $\ca$ is a group homomorphism
\[Z:\go(\ca)\to \C
\]
from the Grothendieck group $\go(\ca)$ of $\ca$ to the additive group of complex numbers such that for each non-zero object $X$ of $\ca$, the number $Z([X])$ is non-zero, where $[X]$ stands for the canonical image of $X$ in the Grothendieck group $\go(\ca)$. The argument $\arg Z([X])$ of $Z([X])$ is called the {\it phase} of $X$ which lies in the interval $[0,\pi)$.  A non-zero object $X$ of $\ca$ is {\it stable} (resp. {\it semi-stable}) if for each non-zero proper subobject $Y$ of $X$, we have $\arg Z([Y])<\arg Z([X])$ (resp. $\arg Z([Y])\leq \arg Z([X])$). The following will be used frequently and implicity.
\begin{lemma}~\label{l:quotient-stable}
Let $Z:\go(\ca)\to \C$ be a stability function on $\ca$.
An object $X\in \ca$ is stable (resp. semi-stable) if and only if for any non-trivial quotient $Y$ of $X$, we have $\arg Z([X])<\arg Z([Y])$(resp. $\arg Z([X])\leq \arg Z([Y])$).
\end{lemma}
\begin{lemma}~\label{l:phase of subobj-quotient}
Let $Z:\go(\ca)\to \C$ be a stability function on $\ca$ and $Y$ be a subobject of $X$ in $\ca$.
 Then
 \begin{itemize}
 \item[$(1)$]$\arg Z([Y])\leq \arg Z([X])$ if and only if $\arg Z([X/Y])\geq \arg Z([X])$;
 \item[$(2)$] if $\arg Z([Y])=\arg Z([X/Y])$, then $\arg Z([X])=\arg Z([Y])=\arg Z([X/Y])$.
 \end{itemize}

\end{lemma}

For each real number $\mu\in [0,\pi)$, let $\ca_{\mu}$ be the full subcategory of $\ca$ whose objects are the zero object and the semi-stable objects of phase $\mu$. It has been proved by King~\cite{King94} that $\ca_{\mu}$  is an abelian subcategory of $\ca$ and the simple objects are precisely the stable objects of phase $\mu$. Moreover, each object $X$ admits a unique filtration called the {\it Harder-Narasimhan filtration}
\[0=X_0\subset X_1\subset\cdots\subset X_t=X
\]
whose subquotients are semi-stable with strictly decreasing phases.

A stability function $Z:\go(\ca)\to \C$ is {\it discrete}, if for each real number $\mu$, the subcategory $\ca_{\mu}$ is zero or admits a unique simple object. In other words, the stability function $Z:\go(\ca)\to \C$ is discrete if and only if different stable objects have different phases.
 If moreover $\ca_{\mu}$ is  semisimple, we then call the stability function $Z:\go(\ca)\to \C$ is {\it completely discrete}.  We mention that for certain abelian categories, this two definitions coincide to each other. This happens for example the category of finitely generated right $kQ$-modules for any Dynkin quiver $Q$.
\begin{remark}
The definition of discrete is weaker than the one given by Keller~\cite{Keller10}, which is equivalent to the definition of completely discrete.  And completely discrete is useful to obtain quantum dilogarithm identities via Ringel-Hall approach.
\end{remark}
\subsection{Discrete stability functions for cyclic quivers}
Let $k$ be a field.
Let $Q:=\Delta_n$ be the cyclic quiver with  vertices set $Q_0=\{1, 2, \cdots, n\}$ such that arrows are going from $i$ to $i+1$ modulo $n$.
Let $S_i, i\in Q_0$ be the simple $k$-representation over the opposite quiver $\Delta_n^{\operatorname{op}}$ associated to the vertex $i$.
Let $\ct_n$ be the category of nilpotent representations over  $\Delta_n^{\operatorname{op}}$. It is a hereditary abelian category such that every indecomposable representation is uniserial and hence uniquely determined up to isomorphism by its unique socle and length. Denote by $R(i,l), i\in Q_0, l\in \N$ the unique indecomposable representation of length $l$ with socle $S_i$. Hence $\{R(i,l)|i\in Q_0, l\in \N\}$ form a completely representatives set of indecomposable objects of $\ct_n$. Let $\go(\ct_n)$ be the Grothendieck group of $\ct_n$. Clearly, we have $\go(\ct_n)\cong \bigoplus_{i\in Q_0}\Z[S_i]$. Set $\delta:=(1,1,\cdots,1)\in \Z^n$. Note that there are exactly $n$ indecomposable objects in $\ct_n$ up to isomorphism with dimension vector $\delta$. For more details of the structure of $\ct_n$, we refer to ~\cite{Ringel93} or also see  Example~\ref{e:AR-structure} below.
\begin{proposition}~\label{p:finite-stable-objects}
Let $Z:\go(\ct_n)\to \C$ be a  stability function on $\ct_n$. Then
\begin{itemize}
\item[$(1)$] each stable object has length  equal or less than $n$ and hence there are only finitely many stable objects;
\item[$(2)$] at most one stable object has dimension vector $\delta$.
\end{itemize}
\end{proposition}
\begin{proof}
The second part follows Lemma~\ref{l:quotient-stable} easily.
For part $(1)$,
let $N\in \ct_n$ be an object of length strictly greater than $n$, then $N=R(i, nt+k)$ for some $i\in Q_0$ and $t\geq 1, 1\leq k<n$ or $N=R(i,nt)$ for $t\geq 2$. If $N=R(i,nt), t\geq 2$, then $N$ admits a proper submodule $R(i,n)$ such that $\arg Z([N])=\arg Z([R(i,n)])$, which implies that $N$ is not stable. If $N=R(i, nt+k)$ for $t\geq 1, 1\leq k<n$, then $R(i,k)$ is a proper submodule and also a quotient module of $N$, which also implies that $N$ is not stable.
\end{proof}
\begin{remark}
By Proposition~\ref{p:finite-stable-objects}, it is not hard to see that there does exist discrete stability functions on $\ct_n$. Indeed, let $\ind \ct_{<n}$ be a representative set of indecomposable objects of length strictly less than $n$.
It is clear that different indecomposable objects in $\ind \ct_{<n}$ have different dimension vectors, we can always choose $Z([S_1]), \cdots, Z([S_n])\in \C^{\times}$ such that $\arg Z([R(1,n)])$ and $\arg Z([M]), M\in \ind \ct_{<n}$ are pairwise different.
\end{remark}

\begin{example}~\label{e:AR-structure}
Let $n=3$. We have the following  Auslander-Reiten quiver of $\ct_3$:
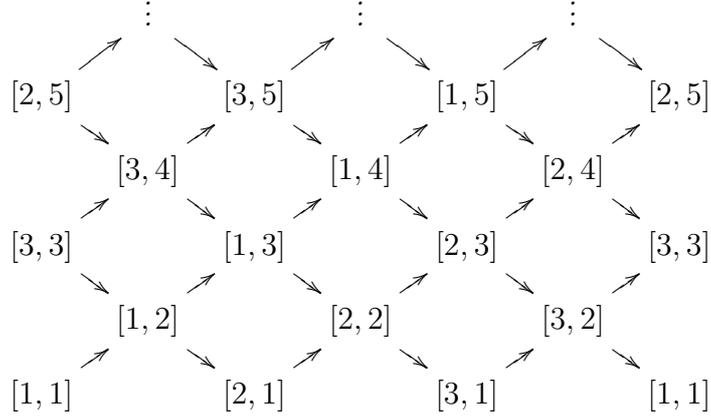
\[\xymatrix@R=0.25cm@C=0.25cm{&~~\vdots~~\ar[dr]&&~~\vdots~~\ar[dr]&&~~\vdots~~\ar[dr]&\\
[2,5]\ar[ur]\ar[dr]&&[3,5]\ar[ur]\ar[dr]&&[1,5]\ar[ur]\ar[dr]&&[2,5]\\
&[3,4]\ar[dr]\ar[ur]&&[1,4\ar[dr]\ar[ur]]&&[2,4]\ar[ur]\ar[dr]&\\
[3,3]\ar[dr]\ar[ur]&&[1,3]\ar[dr]\ar[ur]&&[2,3]\ar[dr]\ar[ur]&&[3,3]\\
&[1,2]\ar[dr]\ar[ur]&&[2,2]\ar[dr]\ar[ur]&&[3,2]\ar[ur]\ar[dr]\\
[1,1]\ar[ur]&&[2,1]\ar[ur]&&[3,1]\ar[ur]&&[1,1]}
\]
where $[i,j]$ stands for the indecomposable object $R(i,j)$.
Since $\go(\ct_3)=\Z[S_1]\oplus \Z[S_2]\oplus \Z[S_3]$, to give a stability function $Z$ on $\ct_3$, it suffices to give the images of $Z([S_i]),i=1,2,3$ over the half plane of $\C$. Let $Z:\go(\ct_3)\to \C$ be the stability function such that $Z([S_1])=2+\sqrt{-1}, Z([S_2])=-2+\sqrt{-1},Z([S_3])=1+2\sqrt{-1}$. It is easy to check that $\arg Z([M]), M\in \ind \ct_{<3}$ and $\arg Z([S_1]+[S_2]+[S_3])$ are pairwise different. Hence  $Z$ is a discrete stability function on $\ct_3$. Moreover, $S_1,S_2,S_3, R[1,2], R[3,3]$ are exactly the stable objects of $\ct_3$.
\end{example}

Let $Z:\go(\ct_n)\to \C$ be a discrete stability function on $\ct_n$ and set $\arg Z(\sum_{i=1}^n[S_i])=\gamma\in [0,\pi)$.
For each $i\in Q_0$, we associate a subset $C_i=\{i, i-1, \cdots, i-t_i\}$ of $Q_0$ (if $x\in \Z$ is not in $Q_0$, we consider its congruence class modulo $n$ in $Q_0$), where $0\leq t_i\leq n-1$, which is the maximum non-negative integer such that
\[\arg Z([S_i]+[S_{i-1}]+\cdots+[S_{i-v}])\geq \gamma ~\text{for all}~ 0\leq v\leq t_i.
\]
Let $Q_0^{\geq \gamma}=\{i\in Q_0|\arg Z([S_i])\geq \gamma\}$ be the subset of $Q_0$. It is obvious that $C_i$ is non-empty if and only if  $i\in Q_0^{\geq \gamma}$.
\begin{lemma}~\label{l:sequence of vertices}
For any $i,j\in Q_0^{\geq \gamma}$, if $C_i\cap C_j\neq \emptyset$, then $C_i\subseteq C_j$ or $C_j\subseteq C_i$.
\end{lemma}
\begin{proof}
Suppose that $C_i\cap C_j\neq \emptyset$ and $s\in C_i\cap C_j$. Then by the definition of $C_i$ and $C_j$, we deduce that $A_i^s:=\{i,i-1, \cdots, s+1, s\}$ is a subset of $C_i$ and $A_j^s:=\{j, j-1, \cdots, s+1,s\}$ is  a subset of $C_j$. Assume that $|A_i^s|\leq |A_j^s|$, then $A_i^s$ is a subset of $A_j^s$. In particular, $i\in C_j$. Now again by the definition of $C_j$, we have $C_i\subseteq C_j$.
\end{proof}
Now we are in the position to state the main result of this section.
\begin{theorem}~\label{t:discrete-structure}
Let $Z:\go(\ct_n)\to \C$ be a discrete stability function on $\ct_n$.
 There is a unique stable object $M\in \ct_n$ such that $\dimv M=\delta=(1,1,\cdots,1)$. In particular, all the nonzero subcategories $(\ct_n)_{\mu}, \mu\in [0,\pi)$, are semisimple  except one which admits $M$ as the unique simple object. As a consequence,
 $\ct_n$ does not admit a completely discrete stability function.

\end{theorem}
\begin{proof}
Assume that $\arg Z(\sum_{i=1}^{n}[S_i])=\gamma\in [0, \pi)$. By Lemma~\ref{l:sequence of vertices}, for any $i,j\in Q_0^{\geq \gamma}$, one of the following cases happens: $C_i\subseteq C_j$;$C_j\subseteq C_i$; $C_i\cap C_j=\emptyset$.
Let $B=\{C_{i_1}, \cdots, C_{i_s}\}$ be a subset of $\{C_{i}|i\in Q_0^{\geq \gamma}\}$ such that for any $1\leq x\neq y\leq s$, $C_{i_x}\cap C_{i_y}=\emptyset$ and $Q_0^{\geq \gamma}\subseteq C_{i_1}\cup\cdots \cup C_{i_s}$.  We claim that $s=1$.  Otherwise, suppose that $s\geq 2$.
For each $C_{i}\in B$, set $\ol{C_i}=C_i\cup \{i-t_i-1\}$. It is clear that $\ol{C_i}\cap\ol{C_j}=\emptyset$ for any $C_i\neq C_j\in B$.
Now we have
\[Z(\sum_{i=1}^{n}[S_i])=Z(\sum_{j\in \ol{C}_{i_1}}[S_j])+\cdots+Z(\sum_{j\in \ol{C}_{i_s}}[S_j])+\sum_{j\not\in \ol{C}_{i_1}\cup\cdots\cup\ol{C}_{i_s}}Z([S_j]).
\]
Recall that $\arg Z(\sum_{i=1}^{n}[S_i])=\gamma$, however the arguments of each term on the right are strictly less than $\gamma$, a contradiction.
 Hence there exists a unique $i_0\in Q_0^{\geq \gamma}$ such that $Q_0^{\geq \gamma}\subseteq C_{i_0}$. By the definition of $C_{i_0}$, one can show that $C_{i_0}=Q_0$ as above. Now the condition of $C_{i_0}$ also implies that $M=R(i_0+1, n)$ is a semi-stable object for $Z$ following Lemma~\ref{l:quotient-stable}.  Note that $Z:\go(\ct_n)\to \C$ is discrete.
  If $M$ is not stable, then it is an iterated extension of certain stable object whose length is strictly less than $n$. However, an indecomposable object with length strictly less than $n$ does not have non-trivial self extension, a contradiction. Hence $M$ is a stable object with non-trivial self extension.
\end{proof}

\section{Quantum dilogarithm identities for cyclic quivers}
\subsection{Hall algebra of $\ct_n$ and quantum torus}
Keep the notations as previous section. Assume moreover that $k$ is a finite field.  For any $L,M,N\in \ct_n$, let $F_{L,M}^N$ be the number of submodules $L'$ of $N$ which are isomorphic to $L$ and such that $N/L'\cong M$.

Let $E$ be a finite field extension of $k$. For any $k$-vector space $V$, let $V^E:=V\otimes_kE$. Clearly, $k\Delta_n^E$ is the path algebra of $\Delta_n$ over $E$. Let $\ct_n^E$ be the category of nilpotent $E$-representations over the opposite quiver $\Delta^{\opname{op}}$. For any $L\in \ct_n$, we have $L^E\in \ct_n^E$.  In fact, the structure of $\ct_n$ is independent of the choice of the base field $k$. For any $M\in \ct_n$, it is clear that we have a polynomial $\Aut_M(q)\in \Z[q]$ such that $|\Aut(M^E)|=\Aut_M(|E|)$ for any finite filed extension $E$ of $k$.

 For any $L,M,N\in \ct_n$, it has been shown by Ringel ~\cite{Ringel93} ({\it cf} also ~\cite{Guo95,Peng-Zhang2001})that there exists a polynomial $\phi_{L,M}^N\in \Z[q]$ such that  for any finite field extension $E$ of $k$,
\[\phi_{L,M}^N(|E|)=F_{L^E,M^E}^{N^E}.
\]
The polynomial $\phi_{L,M}^N$ is called the {\it Hall polynomial} associated to $L,M,N$.

Let $\opname{Iso}(\ct_n)$ be the representative set of isomorphism class of objects in $\ct_n$.
The generic opposite Hall algebra $\ch_q(\ct_n)$ over $\Z[q]$ is a free $\Z[q]$-module with a basis $\{u_L|L\in \opname{Iso(\ct_n)}\}$ index by the isoclass of $\ct_n$. In particular, $\ch_q(\ct_n)=\bigoplus_{L\in \opname{Iso}(\ct_n)}\Z[q]u_L$ and the multiplication is defined as follows
\[u_L*u_M=\sum_{N\in \opname{Iso}(\ct_n)}\phi_{L,M}^N(q)u_N.
\]
It is an associative algebra with unit $u_0$({\it cf.}~\cite{Ringel90}).
Let $\hat{\ch}_q(\ct_n)$ be the completion of $\ch_q(\ct_n)$ with respect to the ideal generated by $u_L$ for all non-zero $L\in \opname{Iso}(\ct_n)$.

Recall that we have the Euler bilinear form $\chi(-,-)$ over $\go(\ct_n)$, for any $M,N\in \ct_n$,
\[\chi([M], [N])=\sum_{i\geq 0}(-1)^i\dim_k\Ext^i_{kQ}(M,N)=\dim_k \Hom_{kQ}(M,N)-\dim_k\Ext^1_{kQ}(M,N).
\]
Let $\lambda(-,-)$ be the anti-symmetric  bilinear form associated to $\chi(-,-)$, {\it i.e.}
\[\lambda([M],[N])=\chi([M],[N])-\chi([N],[M])\quad \text{for any}\quad M,N\in \ct_n.
\]

The {\it completed quantum torus} $\hat{\mathbb{T}}_{\Delta_n}=\Q(q^{\frac{1}{2}})[[\go(\ct_n)^+]]$ is defined to be  the formal power series ring of $y^{\dimv M}, M\in
\ct_n$ with relations $y^{\dimv M}\cdot y^{\dimv
N}=q^{\frac{1}{2}\lambda([M], [N])}y^{\dimv M\oplus N}$.  The following lemma is quite clear.
\begin{lemma}~\label{l:center-torus}
Let $\delta=(1,1,\cdots, 1)$. For any $t\in \N$, $y^{t\delta}$ belongs to the center of $\hat{\mathbb{T}}_{\Delta_n}$.
\end{lemma}
We also have the following integration map from the generic opposite Hall algebra $\ch_q(\ct_n)$ to the quantum torus $\hat{\mathbb{T}}_{\Delta_n}$ due to Reineke~\cite{Reineke2003}.
\begin{lemma}~\label{l:int-map}
There is an algebra homomorphism  $\int:\hat{\ch}_q(\ct_n)\to \hat{\mathbb{T}}_{\Delta_n}$ given by
\[\int u_M=q^{1/2\chi([M],[M])}\frac{y^{\dimv M}}{\opname{Aut}_M(q)}.
\]
\end{lemma}
\vspace{0.5cm}
\subsection{Quantum dilogarithm identities}~\label{ss:quantum-dilogarithm identities}
 Let $q$ be an indeterminant. The {\it
quantum dilogarithm} is the series
\[\E(y)=\sum_{m=0}^{\infty}\frac{q^{\frac{m^2}{2}}}{(q^m-q^{m-1})\cdots(q^m-1)}y^m
\in \Q(q^{\frac{1}{2}})[[y]],
\]
where $\Q(q^{\frac{1}{2}})[[y]]$ is the formal power series ring of
$y$ with coefficients in the rational function field of
$q^{\frac{1}{2}}$.

The following is one of the  main results in this section.
\begin{theorem}~\label{t:quantum-identities}
Let $Z:\go(\ct_n)\to \C$ be a discrete stability function on $\ct_n$. Then
\[\mathbb{E}_{Z}:=\prod_{L ~\text{is stable and } \dimv L\neq \delta}^{\curvearrowright}\mathbb{E}(y^{\dimv L})\in \hat{\mathbb{T}}_{\Delta_n},
\]
where the factors are in the order of decreasing phases, does not depend on the choice of $Z$.
\end{theorem}
\begin{proof}
Let $Z$ be a discrete stability function on $\ct_n$. The Harder-Narasimhan filtration with respect to $Z$ implies the following identity in $\ch_q(\ct_n)$
\[\sum_{L\in \opname{Iso}(\ct_n) }u_L=\prod_{\mu\in [0,\pi)}^{\curvearrowright} (\sum_{L\in \opname{Iso}(\ca_{\mu})}u_L),
\]
where the factor are in the order of  decreasing  $\mu$ and $\ca_\mu:=(\ct_n)_\mu$ is the subcategory of $\ct_n$ with semi-stable objects of phase $\mu$ and the zero objects. Proposition~\ref{p:finite-stable-objects} implies that all but finitely many $\ca_\mu$ are zero, say $\mu_1>\mu_2>\cdots>\mu_t$. Let  $L_i$ be the unique simple object of $\ca_{\mu_i}, i=1, \cdots, t$. Moreover, by Theorem~\ref{t:discrete-structure}, there is a unique $s\in \{1,2,\cdots, t\}$ such that $\ca_{Z,\delta}:=\ca_{\mu_s}$ is not semisimple and $\dimv L_s=\delta$.
We may rewrite the above identity as
\[\sum_{L\in \opname{Iso}(\ct_n) }u_L=\prod_{\mu_1>\mu_2>\cdots> \mu_t}^{\curvearrowright} (\sum_{L\in \opname{Iso}(\ca_{\mu_i})}u_L).
\]
Applying the integration map in Lemma~\ref{l:int-map}, we have
\[\int\sum_{L\in \opname{Iso}(\ct_n) }u_L=\prod_{\mu_1>\mu_2>\cdots> \mu_t}^{\curvearrowright} \int(\sum_{L\in \opname{Iso}(\ca_{\mu_i})}u_L).
\]
Set
\[\mathbb{E}_{Z,\delta}:=\int \sum_{L\in \opname{Iso}(\ca_{\mu_s})} u_L.
\]
By Lemma~\ref{l:center-torus}, we deduce that $\mathbb{E}_{Z,\delta}$ belongs to the center of quantum torus $\hat{\mathbb{T}}_Q$.
On the other hand, for each $\ca_{\mu_i}, i\neq s$, since $\ca_{\mu_i}$ is semisimple wiht unique simple object $L_i$, we have \[\int \sum_{L\in \opname{Iso}(\ca_{\mu_i})}u_L=\mathbb{E}(y^{\dimv L_i}).\]

Hence,
\begin{eqnarray*}
\int\sum_{L\in \opname{Iso}(\ct_n) }u_L&=&\prod_{\mu_1>\mu_2>\cdots> \mu_t}^{\curvearrowright} \int(\sum_{L\in \opname{Iso}(\ca_{\mu_i})}u_L)\\
&=&\mathbb{E}_{Z,\delta}\prod_{i\neq s}^{\curvearrowright}\int(\sum_{L\in \opname{Iso}(\ca_{\mu_i})}u_L)\\
&=&\mathbb{E}_{Z,\delta}\prod_{L ~\text{is stable and } \dimv L\neq \delta}^{\curvearrowright}\mathbb{E}(y^{\dimv L})
\end{eqnarray*}
It is clear that for any discrete stability function, the subcategory $\ca_{Z,\delta}$ is independent of the choice of $Z$ up to equivalence of categories. Thus $\mathbb{E}_{Z,\delta}$ is independent of the choice of $Z$. Note that the left hand side of the above identity is independent of choice of $Z$ and $\mathbb{E}_{Z, \delta}$ is invertible in $\hat{\mathbb{T}}_Q$. Therefore, $\mathbb{E}_Z$ is independent of the choice of $Z$.
\end{proof}

\begin{remark}
Theorem~\ref{t:quantum-identities} implies that the non-commutative Donaldson-Thomas invariant of $\ct_n$ is well-defined({\it cf.} section~4 of ~\cite{Keller10}). Namely, let $\mathbb{T}_{\Delta_n}:=\mathbb{Q}(q^{\frac{1}{2}})[\go^+(\ct_n)]$ be the non-complete quantum torus of $\Delta_n$, {\it i.e.} the $\mathbb{Q}(q^{\frac{1}{2}})$-algebra generated by the variables $y^{\dimv L}, L\in \ct_n$, with the same relations of $\hat{\mathbb{T}}_{\Delta_n}$. Let $\operatorname{Frac}(\mathbb{T}_{\Delta_n})$ be the fraction filed of $\mathbb{T}_{\Delta_n}$ and $\Sigma: \operatorname{Frac}(\mathbb{T}_{\Delta_n})\to  \operatorname{Frac}(\mathbb{T}_{\Delta_n}) $  the automorphism mapping each $y^{\dimv L}$ to $y^{-\dimv L}, L\in \ct_n$. Set $\mathbb{E}_{\Delta_n}:=\mathbb{E}_Z$ for any discrete stability function $Z$ on $\ct_n$ and let $\opname{Ad}(\mathbb{E}_{\Delta_n}):\operatorname{Frac}(\mathbb{T}_{\Delta_n})\to  \operatorname{Frac}(\mathbb{T}_{\Delta_n})$ be the automorphism induced by the conjugation of $\mathbb{E}_{\Delta_n}$. Then the non-commutative Donaldson-Thomas invariant is defined to be the automorphism
\[\opname{DT}_{\Delta_n}=\opname{Ad}(\mathbb{E}_{\Delta_n})\circ \Sigma.
\]

\end{remark}
\vspace{0.5cm}

\subsection{Cyclic quantum dilogarithm  of order $n$}~\label{ss:cyclic quantum}
Let $Q$ be a finite quiver and $\hat{\mathbb{T}}_Q$ the associated completed quantum torus.
A product $\mathbb{E}\in \hat{\mathbb{T}}_Q$ of quantum dilogarithms is called {\it cyclic  of order $n$}({\it cf.}~\cite{BytskoVolkov2013}), if there exists an automorphism $\phi$ of $\hat{\mathbb{T}}_Q$ of order $n$ such that
\[\mathbb{E}=\phi^i(\mathbb{E}), 1\leq i\leq n-1.
\]
 In ~\cite{BytskoVolkov2013}, Bytsko and Volkov have obtained certain cyclic quantum dilogarithm identities of order $3$ for certain quivers. In our setting, we will prove that different choices of discrete stability functions on $\ct_n$ give rise to different cyclic quantum dilogarithm identities of order $n$.

Let $\tau:\ct_n\to\ct_n$ be the Auslander-Reiten translation, which is an auto-equivalence of $\ct_n$ of order $n$. It is clear that $\tau:\ct_n\to \ct_n$ induces an automorphism of Grothendieck group $\go(\ct_n)$ and hence an automorphism of algebra $\hat{\mathbb{T}}_{\Delta_n}$ of order $n$. By abuse of notations, we still denote these automorphisms by $\tau:\go(\ct_n)\to \go(\ct_n)$ and $\tau:\hat{\mathbb{T}}_{\Delta_n}\to \hat{\mathbb{T}}_{\Delta_n}$ respectively. More precisely, $\tau(y^{\dimv M})=y^{\dimv \tau M}$.

For any stability function $Z:\go(\ct_n)\to \C$ on $\ct_n$, we introduce a new stability function $Z_\tau:=Z\circ\tau:\go(\ct_n)\to \C$. We clearly have the followings.
\begin{lemma}~\label{l:tau-stability}
\begin{itemize}
\item[(1)] $Z$ is discrete stability function  if and only if $Z_\tau$ is discrete stability function;
\item[(2)] $M\in \ct_n$ is $Z$-stable if and only if $\tau^{-1}M$ is $Z_\tau$-stable;
\item[(3)] $\arg Z(M)>\arg Z(N)$ if and only if $\arg Z_\tau(\tau^{-1}M)>\arg Z_{\tau}(\tau^{-1}N)$.
\end{itemize}
\end{lemma}
Now we can state another main result in this section.
\begin{theorem}~\label{t:cyclic quantum-identities}
Let $Z$ be a discrete stability function on $\ct_n$. Then we have
\[\mathbb{E}_Z=\tau(\mathbb{E}_Z)=\cdots=\tau^{n-1}(\mathbb{E}_Z)\in \hat{\mathbb{T}}_{\Delta_n}.
\]
\end{theorem}
\begin{proof}
Since $\tau:\hat{\mathbb{T}}_{\Delta_n}\to \hat{\mathbb{T}}_{\Delta_n}$ is an automorphism of $\hat{\mathbb{T}}_{\Delta_n}$ of order $n$, it suffices to prove that $\mathbb{E}_Z=\tau^{n-1}(\mathbb{E})$. Recall that \[\mathbb{E}_Z=\prod_{L ~\text{is $Z$-stable and } \dimv L\neq \delta}^{\curvearrowright}\mathbb{E}(y^{\dimv L}).\]
We have
\begin{eqnarray*}
\tau^{n-1}(\mathbb{E}_Z)&=&\prod_{L ~\text{is $Z$-stable and } \dimv L\neq \delta}^{\curvearrowright}\mathbb{E}(y^{\dimv \tau^{-1}L})\\
&=&\prod_{\tau^{-1}L ~\text{is $Z_\tau$-stable and } \dimv \tau^{-1}L\neq \delta}^{\curvearrowright}\mathbb{E}(y^{\dimv \tau^{-1}L})\\
&=&\mathbb{E}_{Z_\tau}
=\mathbb{E}_{Z}
\end{eqnarray*}
where the second equality follows from Lemma~\ref{l:tau-stability} and the last equality follows from Theorem~\ref{t:quantum-identities}.
\end{proof}

\section{Examples}~\label{s:examples}
\subsection{The Jacobian algebra $J$ of type $A_2$}
Let us consider the cyclic quiver $\Delta_3$ with $3$ vertices:
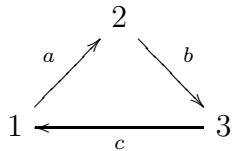
\[\xymatrix{&2\ar[dr]^b\\1\ar[ur]^a&&3\ar[ll]^c}
\]
Let $J:=J(\Delta_3, W)$ be the Jacobian algebra of $\Delta_3$ with potential $E=cba$. In other words, $J=k\Delta_3/<ba, cb, ac>$ is the quotient algebra of path algebra $k\Delta_3$ by the ideal generated by all the paths of  length $2$. Let $\mod J$ be the category of finite-dimensional right $J$-modules, which may be identified as a subcategory of $\ct_3$. Under such identification, we have $\go(\ct_3)=\go(\mod J)$. Hence a stability function $Z:\go(\mod J)\to \C$ on $\mod J$ naturally identifies a stability function $Z:\go(\ct_3)\to \C$ on $\ct_3$. However, a discrete stability function $Z$ on $\mod J$ may not be a discrete stability function on $\ct_3$.
Fortunately, for each discrete stability function $Z:\go(\mod J)\to \C$ on $\mod J$, it is not hard to see that there is a discrete stability function $Z_{\varepsilon}:\go(\mod J)\to \C$ on $\mod J$ such that its sequence of stable objects with decreasing phases coincides with the one of $Z$ and $Z_{\varepsilon}$ is also a discrete stability function on $\ct_3$.

Let $Z:\go(\mod J)\to \C$ be a discrete stability function on $\mod J$ and $Z_{\varepsilon}:\go(\mod J)\to \C$ a replacement of $Z$ which is also a discrete stability function on $\ct_3$ . By Theorem~\ref{t:discrete-structure}, we infer that the set of $Z_{\varepsilon}$-stable objects of $\ct_3$ is the union of the set of $Z_{\varepsilon}$-stable objects (and hence the set of $Z$-stable objects) of $\mod J$  with a unique stable object of $\ct_3$ with dimension vector $(1,1,1)$. Applying Theorem~\ref{t:quantum-identities}, one deduces that the Conjecture $3.2$ of ~\cite{Keller10} holds true for the Jacobian algebra $J$. Namely, for $\mod J$,
 \[\mathbb{E}_Z:=\prod_{M~\text{is stable}}^{\curvearrowright}\mathbb{E}(y^{\dimv M})\in \hat{\mathbb{T}}_{\Delta_3}\]
is independent of the choice of the discrete stability function $Z$ on $\mod J$.
We should mention that by the recent work of Engenhorst~\cite{Engenhorst2013}, Conjecture $3.2$ of ~\cite{Keller10} holds true for any representation-finite Jacobian algebras.

\subsection{Quantum dilogarithm identities of $A_{n-1}$}

Let \[A_{n-1}:1\to 2\to \cdots\to n-1\] be the linear quiver and $\mod kA_{n-1}$ the category of finitely generated right $kA_{n-1}$-modules.

 Recall that by Gabriel's theorem, there is a bijection between the set of the dimension vectors of indecomposable right $kA_{n-1}$-modules to the set of positive root of the associated simple Lie algebra of type $A_{n-1}$. For any positive root $\gamma$, denote by $V(\gamma)$ the unique indecomposable right $kA_{n-1}$-module under this bijection.
We endow the set of positive roots with the smallest order relation such that $\Hom(V(\alpha), V(\beta))\neq 0$ implies $\alpha\preceq \beta$.
Let $\alpha_1, \cdots, \alpha_N$ are the dimension vectors of the indecomposable right $kA_{n-1}$-modules enumerated in decreasing order with $\preceq$.

Identifying $A_{n-1}$ as the full subquiver of $\Delta_n$ supported on vertices $\{1, 2,\cdots, n-1\}$, it induces an embedding of abelian categories  $F:\mod kA_{n-1}\to \ct_n$. Moreover, the functor $F$ also induces an embedding of quantum torus $i_{F}:\hat{\mathbb{T}}_{A_{n-1}}\to \hat{\mathbb{T}}_{\Delta_n}$. Let $Z_1:\go(\ct_n)\to \C$ be a   discrete stability function such that $\arg Z_1([S_1])>\arg Z_1([S_2])>\cdots>\arg Z_1([S_n])$ and $\arg Z_1(\sum_{i=1}^n[S_n])<\arg Z_1([S_{n-1}])$. A directed computation show that $S_1=L(1,1),\cdots, S_{n-1}=L(n-1,1), L(n,n),L(n,n-1), \cdots, S_n=L(n,1)$ are precisely the stable objects with decreasing phases with respect to $Z_1$.
 Note that different indecomposable objects in $\ind \ct_{<n}$ have different dimension vectors. A generically choosing of $Z_2([S_i]), i=1, \cdots, n$ in $\C$ such that $\arg Z_2([S_n])<\arg Z_2([S_1])<\cdots<\arg Z_2([S_{n-1}])$ and $ \arg Z_2(\sum_{i=1}^n[S_n])<\arg Z_2([S_{1}])$ leads $Z_2$ to be a discrete stability function on $\ct_n$. Moreover,
$V(\alpha_1), \cdots, V(\alpha_N), L(n,n),L(n,n-1), \cdots, S_n=L(n,1)$ are precisely the stable objects with decreasing phases with respect to $Z_2$. Applying Theorem~\ref{t:quantum-identities}, we have
\begin{eqnarray*}
&&\mathbb{E}(y^{\dimv S_1})\cdots \mathbb{E}(y^{\dimv S_{n-1}})\mathbb{E}(y^{\dimv L(n,n-1)})\cdots \mathbb{E}(y^{\dimv L(n,1)})\\
&=&\mathbb{E}(y^{\alpha_1})\cdots \mathbb{E}(y^{\alpha_N})\mathbb{E}(y^{\dimv L(n,n-1)})\cdots \mathbb{E}(y^{\dimv L(n,1)}).
\end{eqnarray*}
Since each $\mathbb{E}(y^{\dimv M})$ is invertible in $\hat{\mathbb{T}}_{\Delta_n}$, we have the following quantum dilogarithm identities
\[\mathbb{E}(y^{\dimv S_1})\cdots \mathbb{E}(y^{\dimv S_{n-1}})=\mathbb{E}(y^{\alpha_1})\cdots \mathbb{E}(y^{\alpha_N}).
\]
which  essentially lies in the completed quantum torus $\hat{\mathbb{T}}_{A_{n-1}}$. This recover the well-known quantum dilogarithm identities for $A_{n-1}$({\it cf}. Corollary $1.7$ in ~\cite{Keller10}).

\end{document}